\documentclass{amsart}
\usepackage{amssymb,graphics}

\newcommand{\C}{\mathbb{C}}
\newcommand{\D}{\mathbb{D}}

\newcommand{\N}{\mathbb{N}}

\newcommand{\dom}{\mbox{\rm dom}}

\newcommand{\Int}{\mbox{\rm Int}}
\newcommand{\Ext}{\mbox{\rm Ext}}

\newtheorem{theorem}{Theorem}[section]
\newtheorem{lemma}[theorem]{Lemma}

\theoremstyle{definition}

\theoremstyle{theorem}

\theoremstyle{theorem}
\newtheorem{proposition}[theorem]{Proposition}

\theoremstyle{theorem}

\theoremstyle{theorem}

\theoremstyle{definition}

\theoremstyle{theorem}

\numberwithin{equation}{section}

\begin{document}

\title{Computing links and accessing arcs}

\author{Timothy H. McNicholl}
\address{Department of Mathematics\\
              Iowa State University\\
              Ames, Iowa 50011 USA}
\email{mcnichol@iastate.edu}

\begin{abstract}
Sufficient conditions are given for the computation of an arc that accesses a point on the boundary of an open subset of the plane from a point within the set.  The existence of a not-computably-accessible but computable point on a computably compact arc is also demonstrated.
\end{abstract}

\maketitle                  

\section{Introduction}

Let $\C$ denote the complex plane.  We consider the following situation: we are given an arc $A \subseteq \C$, a point $\zeta_1$ on $A$, and a point $\zeta_0$ that does not lie on $A$.  By the term \emph{arc} we mean a continuous embedding of $[0,1]$ into $\C$.  Such an embedding will then be referred to as a \emph{parameterization} of the arc.  We suppose that we wish to compute a parameterization of an arc $B$ from $\zeta_0$ to $\zeta_1$ that contains no point of $A$ other than $\zeta_1$.  However, we also assume $B$ must be confined to some open set.  The gist of our results is that covering information about $A$ \emph{i.e.} the ability to plot $A$ on a computer screen with arbitrarily good resolution) is not sufficient for the computation of such an arc $B$, but that covering information combined with local connectivity information is. 

Such an arc $B$ is called an \emph{accessing arc}.  More generally, when $\zeta_0$ and $\zeta_1$ are points in the plane, and when $X$ is a subset of the plane, we say that an arc $A$  from $\zeta_0$ to $\zeta_1$ \emph{links $\zeta_0$ to $\zeta_1$ via $X$} if all of its intermediate points belong to $X$.  If $\zeta_0$ is a point in an open set $U \subseteq \C$ and if $\zeta_1$ is a point on the boundary of $U$, then we say that an arc $A$ \emph{accesses $\zeta_1$ from $\zeta_0$ via $U$} if it links $\zeta_0$ to $\zeta_1$ via $U$.   

Our examination of accessing arcs is motivated in part by their relevance to boundary extensions of conformal maps as in \cite{Golusin.1969}, \cite{Nehari.1952}, and \cite{McNicholl.2012}, and to the narrow escape problem in the theory of Brownian motion.  The computation of links between points on the boundary of a domain is the first step in domain decomposition methods such as the Schwarz alternating method \cite{Garnett.Marshall.2005}, \cite{Courant.Hilbert.1989.2}.  In addition to these connections, the problem of computing accessing arcs seems to be an intrinsically interesting problem that admits many intriguing variations such as higher-dimensional versions, computable metric spaces, and rectifiable or computably rectifiable accessing arcs.

Our investigations first lead us to consider the situation in Figure \ref{fig:SITUATION} in which we have an open disk $D$, an arc $A$, a point $\zeta_1$ in $D \cap A$, and a point $\zeta_0$ in $D - A$.  From our computability questions a purely topological question naturally arises.  Namely, how close does $\zeta_1$ have to be to $\zeta_0$ in order for there to be an arc that accesses $\zeta_1$ from $\zeta_0$ via $D - A$?  
An answer is given in Theorem \ref{thm:BOUNDARY.CONNECTED}.   Moreover, the bound in this theorem can be computed from sufficient information about $D$, $\zeta_0$, $\zeta_1$, and $A$.  We then show that when such an accessing arc exists, one of its parameterizations can be computed from sufficient information about $D$, $\zeta_0$, $\zeta_1$, and $A$.  
In particular, local connectivity information about $A$ is used.  

Effective versions of local connectivity are considered in \cite{Brattka.2005}, \cite{Couch.Daniel.McNicholl.2012} and \cite{Daniel.McNicholl.2012}.  In \cite{Brattka.2005}, local connectivity information arises naturally in the consideration of the computational relationships between a function and its graph.  
In \cite{Couch.Daniel.McNicholl.2012}, it is used in the computation of space-filling curves, and in \cite{McNicholl.2012} it is used in the computation of boundary extensions of Riemann maps.  

In Theorem \ref{thm:NO.ACCESS}, we show that mere covering information about the arc $A$ is insufficient for the computation of accessing arcs. 
\begin{figure}[!h]
\resizebox{4in}{4.5in}{\includegraphics{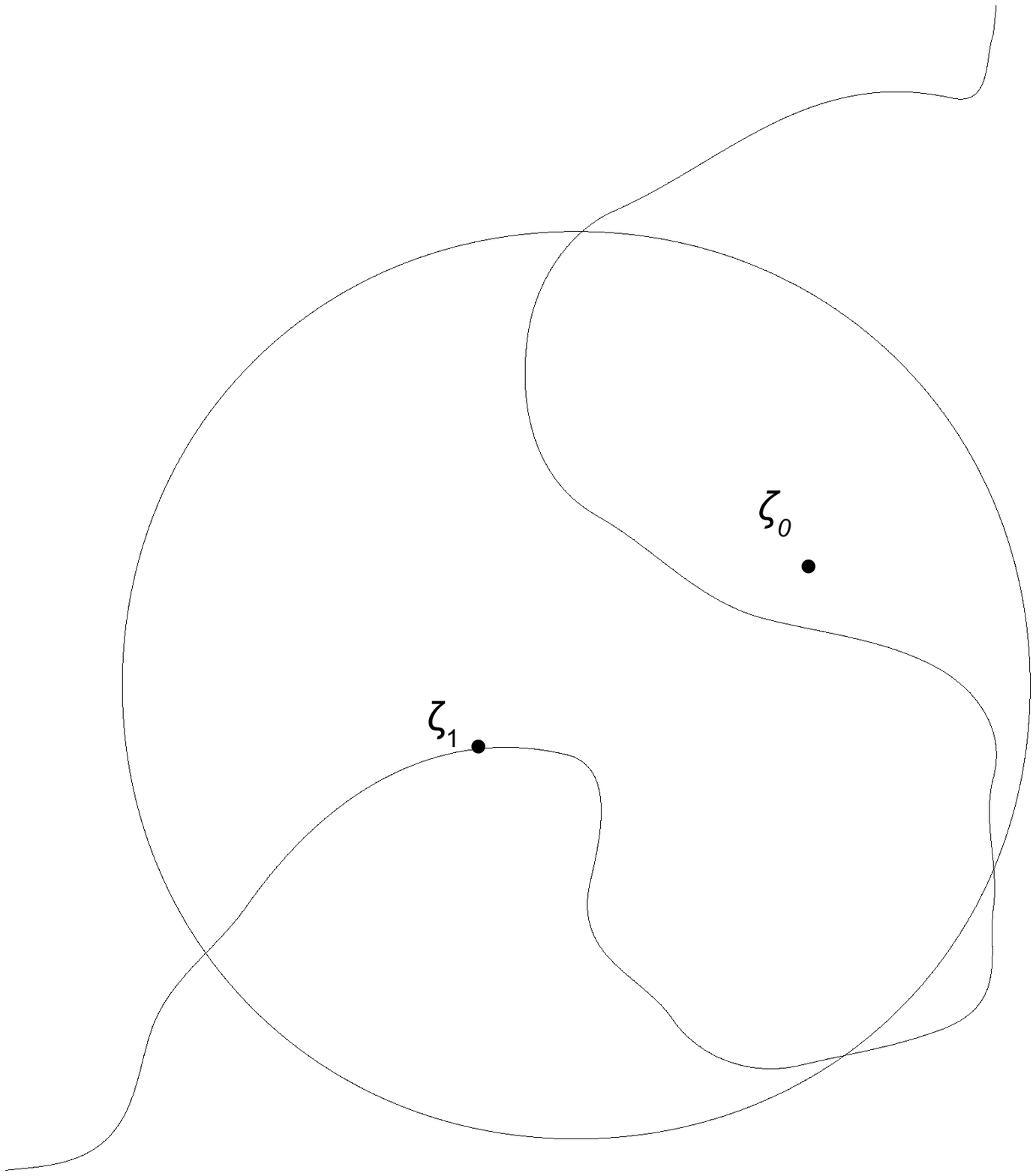}}
\caption{}\label{fig:SITUATION}
\end{figure}

The paper is organized as follows.  Section \ref{sec:BACK.TOP} covers background and preliminaries from topology.  Section \ref{sec:BACK.COMP.A} summarizes the prerequisites from computable analysis.  Section \ref{sec:NO.ACCESS} consists of the proof of Theorem \ref{thm:NO.ACCESS}.  Section \ref{sec:COMPUTING.LINKS} presents the positive results on computing links.
 

\section{Background from topology}\label{sec:BACK.TOP}

When $X, Y \subseteq \C$, let 
\[
d(X,Y) = \inf\{ |z - w|\ :\ z \in X\ \wedge\ w \in Y\}.
\]
Let $d(p, X) = d(\{p\}, X)$ when $p \in \C$ and $X \subseteq \C$.

When $f, g : [0,1] \rightarrow \C$ are bounded, let 
\[
\parallel f - g\parallel_\infty = \sup\{|f(t) - g(t)|\ :\ t \in [0,1]\}.
\]
$\parallel\ \parallel_\infty$ is called the \emph{sup norm}.

Let $f : \subseteq A \rightarrow B$ denote that $f$ is a function whose domain is contained in $A$ and whose range is contained in $B$.

When $f : \subseteq \C \rightarrow \C$, a \emph{modulus of continuity} for $f$ is a 
function $m : \N \rightarrow \N$ such that $|f(z) - f(w)| < 2^{-k}$ whenever $|z - w| \leq 2^{-m(k)}$ and $z,w \in \dom(f)$.  If a function has a modulus of continuity, then it follows that it has an increasing modulus of continuity.  A function has a modulus of continuity if and only if it is uniformly continuous.

Let $D_r(z_0)$ denote the open disk whose radius is $r$ and whose center is $z_0$.
Let $\D = D_1(0)$.

A \emph{curve} is a set $C \subseteq \C$ for which there is a continuous function $f : [0,1] \rightarrow \C$ whose range is $C$.  The function $f$ is called a \emph{parameterization} of the curve $C$.  
The term \emph{parametrization} thus has two different though related uses.  With respect to curves, it refers to a continuous surjection.  But, with respect to arcs it always refers to a continuous bijection.  We will follow the usual custom of identifying a curve and its parameterizations except when computability issues are of concern in which case the distinction is necessary by the results in \cite{Miller.2002}. 

With respect to a particular parameterization $f$ of a curve $C$, if $p = f(0)$ and $q = f(1)$, then the curve $C$ is said to be a curve \emph{from $p$ to $q$}.  

A \emph{cut point} of a set $X \subseteq \C$ is a point $p \in X$ with the property that 
$X - \{p\}$ is disconnected.  The following useful characterization of arcs is an immediate consequence of Theorem 2-27 of \cite{Hocking.Young.1988}.

\begin{proposition}\label{prop:ARC}
A set $A \subseteq \C$ is an arc if and only if it is compact, connected, and has just two non-cut points.  
\end{proposition}

It follows that if $f$ is a parameterization of an arc $A$, then $f(0)$ and $f(1)$ are the non-cut points of $A$.

Let $f : [0,1] \rightarrow \C$ be a curve for which there exist numbers
\[
0 = t_0 < t_1 < \ldots < t_k = 1
\]
and points $v_0, v_1, \ldots, v_k \in \C$ such that 
\begin{eqnarray}
f(x) & = & \frac{x - t_j}{t_{j+1} - t_j}(v_{j+1} - v_j) + v_j\label{eqn:RPC}
\end{eqnarray}
whenever $x \in [t_j, t_{j+1}]$.  $f$ is called a \emph{polygonal curve}.  The points $v_0, \ldots, v_k$ are called the \emph{vertexes} of $f$.  We will call the points $v_1, \ldots, v_{k-1}$ the \emph{intermediate vertexes} of $f$.  A \emph{rational polygonal curve} is a polygonal curve whose vertexes are all rational.  We note that we may take $t_j$ to be $\frac{j}{k}$ in Equation \ref{eqn:RPC}.  

The proof of the following is an easy modification of the proof of Theorem 3.5 of \cite{Hocking.Young.1988}.

\begin{lemma}\label{lm:POLY.ARC}
Suppose $U$ is a domain, and that $p,q$ are distinct points of $U$.  Then, there is a polygonal arc $P$ from $p$ to $q$ that is contained in $U$ and whose intermediate vertexes are rational.   Furthermore, if $\epsilon > 0$, then $P$ can be chosen so that the length of each of its line segments is smaller than $\epsilon$.
\end{lemma}

A \emph{Jordan curve} is a curve that has a parameterization $f$ that is injective except that $f(0) = f(1)$.  When $J$ is a Jordan curve, let $\Int(J)$ denote its interior, and let $\Ext(J)$ denote its exterior. 

The proof of the following is an easy exercise, but it is useful enough to warrant stating it as a proposition.

\begin{proposition}\label{prop:CONNECTED.ACCUMULATION}
If $C \subseteq \C$ is connected, and if $X \subseteq \overline{C}$, then 
$C \cup X$ is connected.
\end{proposition}

\section{Preliminaries from computable analysis}\label{sec:BACK.COMP.A}

Our work is based on the Type Two Effectivity foundation for computable 
analysis which is described in great detail in \cite{Weihrauch.2000}.  
We give an informal summary here of the points pertinent to this paper.
We begin with the naming systems we shall use.  Intuitively, a name of an object 
is a list of approximations to that object that is sufficient to completely identify it.

A \emph{name} of a point $z \in \C$ is a list of all the rational rectangles that contain $z$.

A \emph{name} of a continuous function $f : [0,1] \rightarrow \C$ is a list of rational polygonal curves $P_0, P_1, \ldots$ such that $\parallel P_t - P_s\parallel_\infty \leq 2^{-t}$ whenever $s \geq t$ and $f = \lim_{t \rightarrow \infty} P_t$.  Here, the limit is taken with respect to the supremum norm.   Such a sequence of curves is called a \emph{strongly Cauchy sequence}.  

A \emph{plot} of a compact set $X \subseteq \C$ is a finite set of rational rectangles that each contain a point of $X$ and whose union contains $X$.   
A \emph{name} of a compact $K \subseteq \C$ is a list of all plots of $K$.  These names are called $\kappa_{mc}$-names in 
\cite{Weihrauch.2000}.  They provide precisely the right amount of information necessary to plot the set on a computer screen at any desired resolution.

However, whenever we speak of a name of an arc $A$, we always mean a name of a parameterization of $A$.  And, whenever we speak of a name of a Jordan curve $\gamma$, we always mean a name of a parameterization of $\gamma$, $f$, with the property that $f(s) = f(t)$ only when $s = t$ or $s,t \in \{0,1\}$.   

Once we establish a naming system for a space, an object of that space is called \emph{computable} if it has a computable name.

A sentence of the form 
\begin{quotation}
``From a name of a $p_1 \in S_1$, a name of a $p_2 \in S_2$, $\ldots$, and a name of a $p_k \in S_k$, it is possible to uniformly compute a name of a $p_{k+1} \in S_{k+1}$ such that 
$R(p_1, \ldots, p_k, p_{k+1})$."
\end{quotation}
is shorthand for the following: there is a Turing machine $M$ with $k$ input tapes and a one-way output tape with the property that whenever a name of a $p_j \in S_j$ is written on the $j$-th input tape for each $j \in \{1, \ldots, k\}$ and $M$ is allowed to run indefinitely, a name of a $p_{k+1} \in S_{k+1}$ such that $R(p_1, \ldots, p_{k+1})$ holds is written on the output tape.

A \emph{CIK} (``connected \emph{im kleinen}") function for a set $X \subseteq \C$ is a 
function $f: \N \rightarrow \N$ with the property that whenever $k \in \N$ and $z_0 \in X$, there is a connected set $C \subseteq D_{2^{-k}}(z_0) \cap X$ that contains $D_{2^{-f(k)}}(z_0)  \cap X$.  Related notions are considered in \cite{Daniel.McNicholl.2012}, \cite{Brattka.2008}, \cite{Miller.2002},  and \cite{Couch.Daniel.McNicholl.2012}.

A \emph{ULAC} (``uniformly local arcwise connectivity") function for a set $X \subseteq \C$ is a function $f : \N \rightarrow \N$ with the property that whenever $k \in \N$ and $z_0,z_1$ are distinct points of $X$ such that $|z_0 - z_1| \leq 2^{-f(k)}$, there is an arc $A \subseteq X$ from $z_0$ to $z_1$ whose diameter is smaller than $2^{-k}$.

We will need the following two theorems which follow from the results in \cite{Daniel.McNicholl.2012}.

\begin{theorem}\label{thm:COMPUTE.ARC}
From a name of a compact and connected $C \subseteq \C$, a CIK function for $C$, and names of distinct $\zeta_0, \zeta_1 \in C$, it is possible to compute a name of an arc $A \subseteq C$ from $\zeta_0$ to $\zeta_1$.
\end{theorem}

\begin{theorem}\label{thm:COMPUTE.PARAM}
\begin{enumerate}
	\item From a name of an arc $A \subseteq \C$, it is possible to uniformly compute a name of $A$ as a compact set as well as a CIK function for $A$.
	
	\item From a name of an arc $A \subseteq \C$ as a compact set and a CIK function for $A$, it is possible to uniformly compute a name of $A$.
\end{enumerate}
\end{theorem}

\begin{theorem}\label{thm:ULAC.CIK}
\begin{enumerate}
	\item Every ULAC function is a CIK function.
	
	\item It is possible to uniformly compute, from a name of a compact set $X \subseteq C$ and a CIK function for $X$, a ULAC function for $X$.  
\end{enumerate}
\end{theorem}

\section{The insufficiency of plottability}\label{sec:NO.ACCESS}

\begin{theorem}\label{thm:NO.ACCESS}
The origin belongs to an arc $A$ from $-1$ to $1$ that is computable as a compact set and which has the property that $C \cap (A - \{0\}) \neq \emptyset$ whenever $C$ is a computable curve from $-i$ to $0$.
\end{theorem}

\begin{proof}
We use a diagonalization argument.  
We build $A$ by stages $A_0, A_1, \ldots$.  Each $A_t$ is a polygonal arc with all angles right that goes through $0$.  

Let $S_e = (-2^{-(e+1)}, 2^{-(e+1)})^2$.

Let $\{C_{e,t}\}_{e \in \N, t < k_e}$ be an effective enumeration of all possibly finite, computable, and strongly Cauchy sequences of rational polygonal curves.  If $k_e = \omega$, then let 
$C_e = \lim_t C_{e,t}$.  If $1 \leq k_e < \omega$, then let $C_e = C_{e, k_e - 1}$.  
Otherwise, let $C_e = \emptyset$.  

For each $e$, let $R_e$ be the requirement
\[
R_e : k_e = \omega\ \wedge\ C_e(1) = 0\ \wedge\ C_e(0) \neq 0\ \Rightarrow\ \exists t\ C_e(t) \in A - \{0\}.
\]
\noindent\bf Stage $\mathbf 0$:\rm\ Let $A_0 = [-1,1] \times \{0\}$.  No requirement acts at stage $0$.\\

\noindent\bf Stage $\mathbf t+1$:\rm\ Let us say that $R_e$ \emph{requires attention at stage $t+1$} if after $t$ steps of computation it can be determined that there are rational numbers $0 < t_0 < t_1 < 1$ such that 
\begin{itemize}
	\item $C_e[0, t_0] \cap \overline{S_e} = \emptyset$, 
	\item $C_e[t_0, t_1] \cap S_e \neq \emptyset$,
	\item $C_e[t_1, 1] \subseteq S_e$,
	\item $d(C_e[t_0,t_1], A_t) > 0$, and 
	\item $R_e$ has not acted at any previous stage.
\end{itemize}
If no $R_e$ requires attention at stage $t+1$, then go on to the next stage.  Otherwise, 
let $e$ be the least number such that $R_e$ requires attention at stage $t+1$.  We say that $R_e$ \emph{acts at stage $t+1$}.  
Compute $k \in \N$ such that $k \geq t$, and $2^{-k} < d(C_e[t_0, t_1], A_t)$.
Compute $p_1, p_2 \in (A_t - \overline{S_e}) \cap \bigcap_{e' < e} S_{e'}$ such that $0$ is between $p_1$ and $p_2$ on $A_t$ and the intersection of $S_e$ with the subarc of $A_t$ from $p_1$ to $p_2$ has exactly one connected component.  

Let $q_j$ be a point on $A_t$ between $p_j$ and $0$ such that the subarc of $A_t$ from $p_j$ to $q_j$ lies outside $\overline{S_e}$.\
Let $B$ denote the subarc of $A_t$ from $q_1$ to $q_2$.  We create two parallel copies of $B$, $B_1$ and $B_2$, such that $B$ lies between them and 
\[
B_1 \cup B_2 \subseteq \{z \in \C\ :\ d(z, B) < 2^{-k}\}.
\]
We also construct $B_1$ and $B_2$ so that they contain no point of $A_t$ and so that $B_j \cap S_e$ has only one component for $j = 1, 2$.
  Let $p_{i,j}$ be the endpoint of $B_j$ closest to $p_i$.  

We form $A_{t+1}$ from $A_t$ as follows.  We first remove the subarc of $A_t$ from $q_1$ to $p_1$.  We then add a right angle polygonal arc from $p_1$ to $p_{1,2}$ and the arc $B_2$.  We then remove the subarc from $q_2$ to $p_2$.  We add a right angle polygonal arc from $p_{2,2}$ to $q_2$.  We then add a right angle polygonal arc from $q_1$ to $p_{1,1}$ and the arc $B_1$.  We then add a right angle polygonal arc from $p_{2,1}$ to $p_2$.  

Thus, $S_e - A_{t+1}$ has two more connected components than $S_e - A_t$.  
One of these connected components is bounded by $B_1$, $B$, and the
line segments along the sides of $S_e$ from $B_1$ to $B$.  The other is bounded by 
$B_2$, $B$, and the
line segments along the sides of $S_e$ from $B_2$ to $B$.  Thus, $0$ is a boundary point of each of these components.  However, by the choice of $k$, if $k_e = \omega$, then $C_e$ can not enter either of these components without crossing either $B_1$ or $B_2$.  
If a requirement $R_{e'}$ with $e' < e$ acts at a later stage, its action will further split $B$, $B_1$, and $B_2$, but this will not make things any better for $C_e$.  If a requirement $R_{e'}$ with $e' > e$ acts at a later stage, then $B$ will be further divided, but the situation for $C_e$ will remain the same.    
Thus, $R_e$ is satisfied if it ever acts.  On the other hand, if $C_e$ is a curve from $-i$ to $0$ that contains no point of $A$ but $0$, then $R_e$ must eventually act.  So, every requirement is satisfied. 

It now follows that each requirement is satisfied and that $A =_{df} \lim_t A_t$, where the limit is taken with respect to the Hausdorff metric, is computable as a compact set.  The only non-cut points of $A$ are $-1$ and $1$.  Thus, $A$ is an arc.
\end{proof}

In \cite{Gu.Lutz.Mayordomo.2009}, an arc is constructed that is computable \emph{as a curve} but not as an arc.  That is, it has the property that it is the range of a computable function on $[0,1]$, but is not the range of any computable injective function on $[0,1]$.  Thus, Theorem \ref{thm:NO.ACCESS} is in fact stronger than the assertion that there is no accessing arc.  

\section{Computing links}\label{sec:COMP.LINKS}\label{sec:COMPUTING.LINKS}

We begin with two results which are purely topological but will drive our constructions later.  

\begin{proposition}\label{prop:CONNECTED}
Suppose $\gamma$ is a Jordan curve and that $A \subseteq \overline{\Int(\gamma)}$ is an arc such that at most one endpoint of $A$ belongs to $\gamma$.  Then, $\Int(\gamma) - A$ is connected.
\end{proposition}

\begin{proof}
By the Carath\'{e}odory Theorem (see, \emph{e.g.} Chapter I of \cite{Garnett.Marshall.2005}), we can assume $\gamma = \partial \D$.  Let $p,q \in \D - A$.  
We show there is an arc from $p$ to $q$ in $\D - A$.  By Theorem 4.5 of \cite{Moise.1977}, $\C - A$ is connected.  So, by Lemma \ref{lm:POLY.ARC}, it is also arcwise connected.  Let $B$ be an arc in $\C - A$ from $p$ to $q$.  If $B \subseteq \D$, there is nothing left to prove.  Suppose $B \not \subseteq \D$.  There is a point $p_1 \in B \cap \partial \D$ such that the subarc of $B$ from $p$ to $p_1$ intersects $\partial \D$ only at $p_1$.   There is a point $q_1 \in B \cap \partial \D$ such that the subarc of $B$ from $q$ to $q_1$ intersects $\partial \D$ only at $q_1$.  Hence, $q_1$ is not between $p$ and $p_1$ on $B$.  So, either $p_1 = q_1$ or $q_1$ is between $p_1$ and $q$ on $B$.  Let $B_1$ denote the subarc of $B$ from $p$ to $p_1$.  Let $B_2$ denote the subarc of $B$ from $q$ to $q_1$.
Since $A$ is compact, it follows that there is a point $p_1' \in B_1$ and a point $q_1' \in B_2$ such that $|p_1'| = |q_1'|$ and such that one of the circular arcs from $p_1'$ to $q_1'$ that is concentric with $\D$ contains no point of $A$.  For, otherwise, each subarc of $\partial \D$ from $p_1$ to $q_1$ contains a point of $A$.  Since $p_1, q_1 \not \in A$, these points would be distinct- a contradiction.  It then follows that there is an arc from $p$ to $q$ in $\D - A$.
\end{proof}

\begin{proposition}\label{prop:COMPONENT.ARC}
Let $D$ be an open disk, and let $A$ be an arc.  Let $C$ be a connected component of $D - A$.  Let $p \in A \cap \partial C \cap D$, and suppose $q \in A \cap D - \partial C$.  Then, the subarc of $A$ from $p$ to $q$ intersects the boundary of $D$.
\end{proposition}

\begin{proof}
Let $B$ be the subarc of $A$ from $p$ to $q$.  By way of contradiction, suppose 
$B$ contains no point of the boundary of $D$.   Hence, since $p,q \in D$, $B \subseteq D$.  

Since $D$ is open, there are points $p_1', q_1' \in A$ be such that the subarc of $A$ from $p_1'$ to $q_1'$ is contained in $D$, 
$p$ is between $p_1'$ and $q$ on $A$, and $q$ is between $p$ and $q_1'$ on $A$.  
By Theorem 3-18 of \cite{Hocking.Young.1988}, there are points 
$p_1$ and $q_1$ on $A$ and points $p_2$, $q_2$ in $D - A$ such that 
$p_1$ is between $p_1'$ and $p$ on $A$, $q_1$ is between $q$ and $q_1'$ on $A$, 
$\overline{p_2p_1} \cap A = \{p_1\}$\footnote{Here, and elsewhere expressions of the form $\overline{zw}$ refer to the line segment from $z$ to $w$, not to the conjugate of $zw$.}, and $\overline{q_2q_1} \cap A = \{q_1\}$.  Let 
$B_1$ be the subarc of $A$ from $p_1$ to $q_1$.  By Proposition \ref{prop:CONNECTED}, $D - B_1$ is connected.

By Lemma \ref{lm:POLY.ARC}, there is a polygonal arc $P \subseteq D - B_1$ from $p_2$ to $q_2$.
It follows that there is an arc $\sigma \subseteq P \cup \overline{p_2p_1} \cup \overline{q_2q_1}$ from $p_1$ to $q_1$.  
(Namely, follow $\overline{p_1p_2}$ until $P$ is first reached, then follow $P$ until 
$\overline{q_2q_1}$ is first reached after which $\overline{q_2q_1}$ is followed until $q_1$ is reached.)  Hence, 
$\sigma \cap B_1 = \{p_1, q_1\}$.  Thus, $J =_{df} B_1 \cup \sigma$ is a Jordan curve.

We first consider the case where there are points of $C \cap \Ext(J)$ arbitrarily close to $p$.
Let $f$ be a conformal map of $D_1 =_{df} D - \overline{\Int(J)}$ onto an annulus 
$G =_{df} \{z\ |\ r_1 < |z| < r_2\}$.  By Theorem 15.3.4 of \cite{Conway.1995}, $f$ extends to a homeomorphism of $\overline{D_1}$ with $\overline{G}$; let $f$ denote this extension as well.  We can assume
$f$ maps $J$ onto the inner circle of $G$.  
It follows that $f(p), f(q) \in f[B_1] \subseteq \partial D_{r_1}(0)$.  Let $f(p) = r_1 e^{i \theta_1}$, and let
$f(q) = r_1 e^{i \theta_2}$.  Without loss of generality,  suppose $0 < \theta_1 < \theta_2 < 2\pi$.  We claim there is an $R > r_1$ and an $\epsilon > 0$ such that 
\[
\{r e^{i \theta}\ |\ \theta_1 - \epsilon < \theta < \theta_2 + \epsilon\ \wedge\ r_1 < r < R\} - f[A]
\]
is connected.  For, otherwise, there are points of $f[A - B_1]$ that are arbitrarily close to $f[B]$.  This entails that $B \cap (\overline{A - B_1}) \neq \emptyset$ which violates the assumption that $A$ is an arc.  Since $C \cap \Ext(J)$ contains points arbitrarily close to $p$, it now follows that $q$ is a boundary point of $C$.

If there are points of $C \cap \Int(J)$ arbitrarily close to $p$,  then we proceed similarly except we first conformally map $\Int(J)$ onto $\D$.

Suppose by way of contradiction that neither of these cases holds.  Then, there is a positive number 
$\epsilon$ such that $D_\epsilon(p)$ contains no point of $C \cap \Ext(J)$ nor any point of $C \cap \Int(J)$.  Let $\epsilon_1$ be a positive number that is smaller than $\epsilon$ and that has the property that $D_{\epsilon_1}(p) \cap \sigma = \emptyset$.  Let $w$ belong to $D_{\epsilon_1}(p) \cap C$.  Thus, $w \in J$.  Hence, $w \in B_1 \subseteq A$; this is a contradiction since $C \subseteq D - A$.
\end{proof}

The following answers the first question raised in the introduction.

\begin{theorem}\label{thm:BOUNDARY.CONNECTED}\label{thm:BOUNDARY.COMPONENT}
Suppose $D$ is an open disk, $A$ is an arc with ULAC function $g$, and $\zeta_0 \in A \cap D$.  Suppose $\zeta_1 \in D - A$ is such that $|\zeta_0 - \zeta_1| < 2^{-g(k)}$ where $k \in \N$ is such 
that $2^{-g(k)} + 2^{-k} \leq \max\{d(\zeta_0, \partial D), d(\zeta_1, \partial D)\}$.  Then, $\zeta_0$ is a boundary point of the connected component of $\zeta_1$ in $D - A$.  
\end{theorem}

\begin{proof}
Let $l = \overline{\zeta_1\zeta_0}$.  If $l \cap A = \{\zeta_0\}$, then there is nothing left to prove.
So, suppose $l \cap A \neq \{\zeta_0\}$.  Let $p$ be the point in $l \cap A$ that is closest to $\zeta_1$.  Let $C$ be the connected component of $\zeta_1$ in $D - A$.  Hence, $p \in \partial C$.  Let $A_1$ be the subarc of $A$ from $p$ to $\zeta_0$.  Since $|p - \zeta_0| < 2^{-g(k)}$, the diameter of $A_1$ is smaller than $2^{-k}$.

We claim that $A_1 \subseteq D$.  For, suppose otherwise, and let $q \in \partial D \cap A_1$.  Hence, $|\zeta_0 - q| < 2^{-k}$.  Thus, $d(\zeta_0, \partial D) < 2^{-k} < 2^{-g(k)} + 2^{-k}$.  At the same time, 
\begin{eqnarray*}
|\zeta_1 -q| & \leq & |p - \zeta_1| + |p - q|\\
& < & 2^{-g(k)} + 2^{-k}.
\end{eqnarray*}
Hence, $d(\zeta_1, \partial D) < 2^{-g(k)} + 2^{-k}$.  This is a contradiction since 
$2^{-g(k)} + 2^{-k} \leq \max\{ d(\zeta_0, \partial D), d(\zeta_1, \partial D)\}$.  Hence, 
$A_1 \subseteq D$.

It now follows from Proposition \ref{prop:COMPONENT.ARC} that $\zeta_0$ is a boundary point of $C$.
\end{proof}

We now turn to the problem of computing accessing arcs.

\begin{theorem}\label{thm:ACCESS}
From a name of an arc $A$, a point $z_0 \in \D - A$, and a name of a point $\zeta_0 \in A \cap \D$ that is a boundary point of the connected component of $z_0$ in $\D - A$, it is possible to uniformly compute a name of an arc $Q$ that links $z_0$ to $\zeta_0$ via $\D - A$.
\end{theorem}

\begin{proof}
Compute an increasing
ULAC function for $A$, $g$.  Compute $s_0 \in \N$ such that $D_{2^{-s_0 + 2}}(\zeta_0) \subseteq \D$ and so that $2^{-s_0 + 2} < |z_0 - \zeta_0|$.    

Since $\zeta_0$ is a boundary point of the connected component of $q_0$ in $\D - A$, there is a rational point $e_0$ in this component such that 
$|e_0 - \zeta_0| < 2^{-g(s_0)}$.  It follows from Theorem 3-2 of \cite{Hocking.Young.1988} that this component is open.  Hence, there is a polygonal arc $P_0$ from 
$z_0$ to $e_0$ contained in $\D - A$.    It follows from Lemma \ref{lm:POLY.ARC} that such a point $e_0$ and such an arc $P_0$ can be discovered by a search procedure.  Namely, we search for distinct rational points $q_1, \ldots, q_k \in \D - A$ that satisfy the following conditions.
\begin{enumerate}
	\item $q_j \neq z_0$ when $j \in \{1, \ldots, k\}$.
	\item $|q_k - \zeta_0| < 2^{-g(s_0)}$.
	\item $\overline{z_0q_1} \cap \overline{q_1q_2} = \{q_1\}$.\label{C}
	\item $\overline{z_0q_1} \cap \overline{q_jq_{j+1}} = \emptyset$ when $1 < j < k$.
	\item $\overline{q_jq_{j+1}} \cap \overline{q_{j+1}q_{j+2}} = \{q_{j+1}\}$ when $1 \leq j < k - 1$.
	\item $\overline{q_jq_{j+1}} \cap \overline{q_mq_{m+1}} = \emptyset$ when $m > j+1$.
\end{enumerate}
Condition \ref{C} can be checked by checking that $\min\{d(z_0, \overline{q_1q_2}), d(q_2, \overline{z_0q_1})\} > 0$.  By Lemma \ref{lm:POLY.ARC}, we can also choose $q_1$ so that $|z_0 - q_1| < |z_0 - \zeta_0| - 2^{-s_0 + 2}$.  Thus, $\overline{z_0q_1}$ contains no point of the closed disk with center $\zeta_0$ and radius $2^{-s_0 + 2}$.

Now, by way of induction, suppose $|e_t - \zeta_0| < 2^{-g(s_t)}$, $s_t \geq t, s_0$.  
Let $\epsilon_t = 2^{-g(s_t)} + 2^{-s_t}$.   
We first note that 
\[
D_{\epsilon_t}(e_t) \subseteq D_{2^{-s_t + 2}}(\zeta_0).
\]
Since $s_t \geq s_0$, and since $D_{2^{-s_0 + 2}}(\zeta_0) \subseteq \D$, it follows that $D_{\epsilon_t}(e_t) \subseteq \D$.

Compute $s_{t+1} > \max\{s_t, t+1\}$ such that $d(\zeta_0, \bigcup_{s \leq t} P_s) > 2^{-s_{t+1} + 2}$.  
It follows from Theorem \ref{thm:BOUNDARY.CONNECTED} that $\zeta_0$ is a boundary point of the connected component of $e_t$ in $D_{\epsilon_t}(e_t) - A$.  Hence, there is a rational point $e_{t+1}$ that belongs 
to this connected component such that 
$|e_{t+1} - \zeta_0| < 2^{-g(s_{t+1})}$.    Since this component is open, there is a rational polygonal arc $P_{t+1}$ from $e_t$ to $e_{t+1}$ such that $P_{t+1} \subseteq D_{\epsilon_t}(e_t) - A$.  It follows from Lemma \ref{lm:POLY.ARC} that such a point $e_{t+1}$ and such an arc $P_{t+1}$ can be discovered through a search procedure.  Note that $P_{t+1} \subseteq D_{2^{-s_t + 2}}(\zeta_0)$.

Note that by construction, $P_1$ contains no point of $\overline{z_0q_1}$.  Therefore, for each $j \in \N$ we can compute the least $t_j$ such that $P_j(t_j)$ belongs to $P_{j+1}$.  Note that $t_j$ and $P_j(t_j)$ are rational.  By construction, $z_0 \neq P_0(t_0)$ and $P_j(t_j) \neq P_{j+1}(t_{j+1})$.  Let $Q_0$ be the subarc of $P_0$ from $z_0$ to $P_0(t_0)$.  Let $Q_{j+1}$ be the sub arc of $P_{j+1}$ from $P_j(t_j)$ to $P_{j+1}(t_{j+1})$.  
Define $Q(1)$ to be $\zeta_0$.  When $\frac{j}{j+1} \leq t \leq \frac{j+1}{j+2}$, define $Q(t)$ to be $Q_j(s)$ where
\[
s = \frac{t - \frac{j}{j+1}}{\frac{j+1}{j+2} - \frac{j}{j+1}}.
\]

It follows that $Q$ can be uniformly computed from the given data.  By construction, $Q \cap A = \{\zeta_0\}$.
\end{proof}

Finally, we turn to the problem of computing links between points on the boundary of a connected open set.   We provide an answer when the boundary is locally arc-like.  For example, when the boundary is a union of disjoint Jordan curves.

\begin{theorem}\label{thm:COMPUTING.LINK}
From a name of an open and connected $D \subseteq \C$, names of distinct points $\zeta_0, \zeta_1 \in \partial D$, arcs $B_0, B_1$, and a rational number $r > 0$ such that $D_r(\zeta_j) \cap \partial D \subseteq B_j$, it is possible to compute an arc $A$ that links $\zeta_0$ to $\zeta_1$ via $D$.
\end{theorem}

\begin{proof}
Without loss of generality, we can assume $D_r(\zeta_0) \cap D_r(\zeta_1) = \emptyset$.  
Compute an increasing ULAC function for $B_j$, $g_j$.  Compute a number $k \in \N$ such that 
$2^{-k + 1} \leq r = d(\zeta_j, \partial D_r(\zeta_j))$.  
For each $j$, compute a rational point $\xi_j \in D - B_j$ such that $|\xi_j - \zeta_j| < 2^{-g_j(k)}$.
By Theorem \ref{thm:BOUNDARY.CONNECTED}, $\zeta_j$ is a boundary point of the connected component of $\xi_j$ in $D_r(\zeta_j) - B_j$.  Therefore, by Theorem \ref{thm:ACCESS}, it is possible to uniformly compute from the given data an arc $A_j \subseteq D_r(\zeta_j)$ from $\xi_j$ to $\zeta_j$ such that $A_j \cap B_j =\{\zeta_j\}$.  Hence, 
$A_j \cap \partial D = \{\zeta_j\}$.  Furthermore, $A_1 \cap A_2 = \emptyset$.  By Lemma \ref{lm:POLY.ARC}, we can compute a rational polygonal arc $P \subseteq D$ from $\xi_1$ to $\xi_2$ from the given data.  It may be that $P$ has one or more points in common with $B_1$ besides $\xi_1$, and it may have one or more points in common with $B_2$ besides $\xi_2$.  However, by using the techniques in the proof of Theorem \ref{thm:ACCESS}, we can cull an arc $A$ from $B_1 \cup P \cup B_2$ as required.
\end{proof}

\section*{Acknowledgement}
I thank the referee for his valuable comments.  I also thank Jack Lutz for helpful and stimulating conversations on these topics.  I also thank my wife Susan for support.

\bibliographystyle{amsplain}
\def\cprime{$'$}
\providecommand{\bysame}{\leavevmode\hbox to3em{\hrulefill}\thinspace}
\providecommand{\MR}{\relax\ifhmode\unskip\space\fi MR }
\providecommand{\MRhref}[2]{%
  \href{http://www.ams.org/mathscinet-getitem?mr=#1}{#2}
}
\providecommand{\href}[2]{#2}

\end{document}